\documentclass[10pt]{article}
\usepackage{amsmath, amsthm,amssymb}
\usepackage{bm}
\usepackage{bbm}
\usepackage{tensor}
\usepackage[hidelinks]{hyperref} 
\usepackage{makeidx}
\usepackage{enumerate}
\usepackage{stmaryrd}
\usepackage{fancyhdr}
\usepackage{tikz,color,enumitem}
\numberwithin{equation}{section}
\title{Peripheral Poisson boundaries and jointly bi-harmonic functions}
\author{Sayan Das}

\date{}

\newtheorem{thm}{Theorem}[section]
\newtheorem{thma}{Theorem}

\newtheorem{prop}[thm]{Proposition}
\newtheorem{cor}[thm]{Corollary}
\newtheorem{lem}[thm]{Lemma}
\newtheorem{examp}[thm]{Example}

\newcommand{\BB}{\operatorname{\mathcal B}} 
\newcommand{\G}{{\Gamma}} 
\newcommand{\g}{{\gamma}} 
\newcommand{\HH}{{\mathcal H}} 
\newcommand{\mcp}{{\mathcal P}} 

\newcommand{\email}{Email: } 

\newcommand{\Har}{\operatorname{Har}}

\begin{document}

%
%

\maketitle
\begin{abstract}
	In this paper we answer a question of Kaimanovich by characterizing (jointly) bi-harmonic functions on countable, discrete groups with respect to a symmetric, generating measure. We also study the peripheral Poisson boundary of $L(\G)$ with respect to Markov operators arising from symmetric, generating probability measures on a countable, discrete group $\G$. We solve a recent conjecture of Bhat, Talwar and Kar regarding peripheral eigenvalues and their corresponding eigenvectors for such Markov operators, and provide a complete description of the peripheral Poisson boundary in the aforementioned scenario. 
\end{abstract}
\section{Introduction}
Let $\G$ be a countable, discrete group, and let $\mu$ be a probability measure on $\G$. In early 60's, Furstenberg found a natural way to associate a $\G$-measure space $(\BB, \beta)$ that captures the asymptotic properties of the (right)-random walk on $\G$ with law $\mu$ \cite{ furstenbergproducts, furstenberg}. This measure space is called the Poisson boundary of $\G$ with respect to the measure $\mu$. The associated Markov operator $\mcp_{\mu}: \ell^{\infty}(\G) \rightarrow \ell^{\infty}(\G)$ is defined by  
\[
\mcp_{\mu}(f)(g)= (f \ast \mu)(g)= \sum\limits_{h \in \G} \mu(h)f(gh), \text { where } f \in \ell^{\infty}(\G), \text{ and } g \in \G.
\]
The space of bounded harmonic functions is defined as the fixed points of the Markov operator $\mcp_{\mu}$, and is denoted by $\Har(\mu)$. There is a one-to-one correspondence between bounded harmonic functions and (essentially) bounded measurable functions on the associated Poisson boundary$(\BB, \beta)$. While the space of bounded harmonic functions is not closed under pointwise product in general, one may define a new product of bounded harmonic functions that turns $\Har(\mu)$ into an an abelian von Neumann algebra as follows. Let $f_1, f_2 \in \Har(\mu)$. Then $f_1 \diamond f_2 = \lim \limits_{n \rightarrow \infty}(f_1 f_2)\ast \mu^{\ast n}$, where the limit exists pointwise. Equipped with this new multiplication, we have the isomorphism of von Neumann algebras $\Har(\mu) \cong L^{\infty}(\BB, \beta)$. \vskip 0.01 in
The Poisson boundary enjoys remarkable rigidity properties as a $\G$-space and has become a major object of study for establishing rigidity properties of groups and their probability measure preserving actions \cite{margulissuperrigidity, zimmersuperrigidity, badershalom, burgermonod, baderfurman}.
Similarly, one may consider the left-random walk on $\G$, with the corresponding Markov operator given by
\[
\mcp_{\mu}^o(f)(g)= (\mu \ast f)(g)= \sum\limits_{h \in \G} \mu(h)f(hg), \text { where } f \in \ell^{\infty}(\G), \text{ and } g \in \G.
\]
Clearly one gets analogous theories whether one considers right, or left random walks on $\G$. 
\vskip 0.01 in
Despite the plethora of results concerning $\mu$-harmonic functions, the study of bi-harmonic functions was only sparsely considered, as noted by Kaimanovich in \cite[Section II]{kaimanovich1}. For a symmetric, generating probability measure $\mu$, separately $\mu$-bi-harmonic functions, i.e. functions $f \in \ell^{\infty}(\G)$ with $f \ast \mu= \mu \ast f= f$ were studied by Raugi \cite{raugi} and Willis \cite{willis}. In particular, Willis showed that these functions are constant. The fact that bounded separately $\mu$-bi-harmonic functions are constant is also equivalent to Kaimanovich's seminal double ergodicity theorem.
\vskip 0.01 in
  Motivated by this, Kaimanovich \cite{kaimanovich1} considered the case of bounded jointly $\mu$-bi-harmonic functions, i.e., functions $f \in \ell^{\infty}(\G)$ satisfying $\mu \ast f \ast \mu =f$. To the best of our knowledge, this is the only paper in the literature dealing with jointly $\mu$-bi-harmonic functions. Kaimanovich proved \cite[Theorems 3 and 4]{kaimanovich1} that bounded jointly $\mu$-bi-harmonic functions must be separately $\mu$-bi-harmonic under left and right convolution, and hence constant. However, Kaimanovich's proof has a gap. As we show in Example~\ref{eg: anti-harmonic for Z} and Example~\ref{eg: anti-harmonic for F_2}, there exist jointly $\mu$-bi-harmonic functions that are not constant! Roughly speaking, these jointly $\mu$-bi-harmonic functions are (separately) anti $\mu$-harmonic under left and right convolutions. 
Our first main theorem characterizes jointly $\mu$-bi-harmonic functions by showing that they are separately anti $\mu$-harmonic, upto a constant. Hence, this result completely answers Kaimanovich's question regarding jointly $\mu$-bi-harmonic functions that he considered in \cite{kaimanovich1}.
\begin{thma}
	Let $\mu$ be a symmetric, generating probability measure on $\G$. Let $f \in \ell^{\infty}(\G)$ be a jointly $\mu$-bi-harmonic function, i.e. $\mu \ast f \ast \mu=f$. Then there exists $c \in \mathbb C$ such that $f-c$ is (separately) anti-harmonic under both left and right convolution by $\mu$.
\end{thma}

This naturally leads to the investigation of anti-harmonic functions on $\G$. Our second main theorem is a characterization of anti-harmonic functions in terms of harmonic functions and anti-harmonic characters. This solves a recent conjecture of Bhat, Talwar and Kar (see the last paragraph of \cite[Example 3.4]{BTK22}).
\begin{thma}
	Let $\mu$ be a symmetric, generating measure on a countable, discrete group $\G$. Assume that there exists a nonzero function $f \in \ell^{\infty}(\G)$ with $f \ast \mu =-f$. Then there exists a multiplicative character $\chi: \G \rightarrow \mathbb T$ such that $\chi|_{\text{supp}(\mu)} \equiv -1$, and hence $\chi \ast \mu = - \chi$. \vskip 0.01 in
	If $F$ is any anti-harmonic function, then we can find a harmonic function $f_1$ such that $F=f_1 \cdot \chi$. 
\end{thma}
Due to the remarkable applications of the study of Poisson  boundaries to rigidity phenomenon in the group case, A. Connes suggested that one should try to understand the notion of noncommutative Poisson boundaries of $\rm II_1$ factors (see \cite[Page 86]{jonesproblems}). A major step in this direction was undertaken by Izumi in early 2000's \cite{izumi2, izumi3, izumi4}. Izumi defined the Poisson boundary corresponding to normal, unital, completely positive (ucp) maps acting on von Neumann algebras by showing that the fixed point space of such maps is a weakly closed injective operator system, and hence can be equipped with the Choi-Effros product \cite{choieffros} to endow it with a von Neumann algebra structure. Later Prunaru \cite{prunaru} and Izumi \cite[Appendix]{izumi3} (the observation is credited to Bill Arveson) independently observed that this von Neumann algebra coincides with the fixed points of the minimal dilation of the ucp map under consideration (see also \cite[Appendix]{DP22}). 
\vskip 0.01 in
Noncommutative Poisson boundaries (for $L(\G)$) played a key role in Peterson's seminal work on character rigidity \cite{petersonchar} (see also \cite{creutzpetersonchar, creutzpetersonfree}). Motivated by the aforementioned discussion, Prof.\ Jesse Peterson and the author initiated the study of Poisson boundaries of a finite von Neumann algebra $M$. Given an $M$-bimodular normal ucp map $\mcp_{\varphi}: \BB(L^2(M)) \rightarrow \BB(L^2(M))$, we considered the corresponding boundary of the ucp map in Izumi's sense. The main difference between our work and Izumi's is that Izumi's work should be considered as a noncommutative generalization of Poisson boundaries of general Markov operators, while we were motivated to find a notion of Poisson boundary that fits closer to the group case. This distinction of perspectives will also show up in the current work.
\vskip 0.01 in
Motivated by Izumi's perspective Bhat, Talwar and Kar \cite{BTK22} recently considered the notion of Peripheral Poisson boundary. Given a normal ucp map $\Phi: M \rightarrow M$, where $M$ is a von Neumann algebra, a unimodular complex number $\lambda \in \mathbb T$ is called a \emph{peripheral eigenvalue} of $\Phi$ if there exists $0 \neq x \in M$ with $\Phi(x)=\lambda x$. Using Bhat's minimal dilation \cite{bhatdilation}, they showed that the norm closed span of all peripheral eigenspaces form a $C^*$-algebra, called the \emph{Peripheral Poisson boundary}, under a new Choi-Effros type product. Remarkably, the Peripheral Poisson boundary need not be a von Neumann algebra in general \cite[Page 14]{BTK22}.
\vskip 0.01 in
Given a probability measure $\mu$ on $\G$ whose support generates $\G$ as a semigroup, one may consider the extension of the Markov operators $\mcp_{\mu}$ to $\BB(\ell^2(\G))$ as follows
\[
\mcp_{\mu}(T)= \sum\limits_{g \in \G} \mu(g) \rho_g T \rho_g^*, \text{ where } T \in \BB(\ell^2(\G)), \text{ and } \rho \text{ denotes the right regular representation. }
\]  
Izumi \cite{izumi4} showed that the noncommutative Poisson boundary of the ucp map $\mcp_{\mu}$ is the crossed product von Neumann algebra $L^{\infty}(\BB, \beta) \rtimes \G$, where $(\BB, \beta)$ is the classical Poisson boundary of $(\G, \mu)$. In \cite[Example 3.4]{BTK22} the authors raised the question regarding the peripheral eigenvalues corresponding to the ucp map $\mcp_{\mu}$. In our third main Theorem we show that all peripheral eigenvalues are roots of unity, and hence the Peripheral Poisson boundary is a von Neumann algebra in this case.
\begin{thma}
	Let $\mu$ be a probability measure  on a countable, discrete group $\G$, such that the support of $\mu$ generates $\G$ as a semigroup. Consider the noncommutative Markov operator $\mcp_{\mu}: \BB(\ell^2(\G)) \rightarrow \BB(\ell^2(\G)) $ given by $\mcp_{\mu}(T)= \sum_g \mu(g)\rho_g T \rho_g^*$. Let $\lambda \in \mathbb T$ be a peripheral eigenvalue for $\mcp_{\mu}$. Then, $\lambda^k =1$ for some positive integer $k$. \vskip 0.01 in
	In particular, the Peripheral Poisson boundary is a von Neumann algebra in this case. \vskip 0.01 in
	If $\mu$ is symmetric, then the only possible peripheral eigenvalues are $1$ and $-1$.
\end{thma}
 
\section{Peripheral Eigenvalues}
Throughout this section we will assume that $\mu$ is a probability measure on a countable discrete group $\G$. The corresponding Markov operator on $\BB(\ell^2(\G))$ will be denoted by $\mcp_{\mu}$.
The following Proposition was communicated to the author by Prof.\ Jesse Peterson. The author would like to thank Prof.\ Peterson for allowing him to reproduce the argument here.
\begin{prop} \label{prop: eigenoperator to eigenfunction}
	Let $\mu$ be a probability measure on $\G$. Let $\eta \in \mathbb T$ be a peripheral eigenvalue of $\mcp_{\mu}$, and let $0 \neq T \in \BB(\HH)$ be a corresponding eigenvector, i.e., $\mcp_{\mu}(T)= \eta T$. Then there exists a nonzero $f \in \ell^{\infty}(\G)$ such that $f \ast \mu =\eta f$. In particular, if $\mu$ is a symmetric, generating probability measure on $\G$ then the Markov operator $\mcp_{\mu}$ admits $-1$ as an eigenvalue if and only if there exists a nonzero anti-$\mu$-harmonic function.
\end{prop}

\begin{proof}
	Consider the Fourier expansion $T= \sum_{g \in \G}\mathcal E(T \lambda_g^*)\lambda_g$, where $\mathcal E: \BB(\HH) \rightarrow \ell^{\infty}(\G)$ is the canonical normal conditional expectation. Note that $\mathcal E \circ \mcp_{\mu}= \mcp_{\mu} \circ \mathcal E$. Hence, for all $g \in \G$ we have $ \mcp_{\mu}(\mathcal E(T\lambda_{g}^*))=\mathcal E(\mcp_{\mu}(T \lambda_g^*))=\mathcal E(\mcp_{\mu}(T) \lambda_g^*)=\eta \mathcal E(T \lambda_g^*)$. Since $T$ is nonzero, we can find $g \in \G$ such that $f := \mathcal E(T \lambda_g^*) \neq 0$. Hence, $f \ast \mu = \eta f$.	
\end{proof}
We will provide a proof of Theorem C. We begin with the case when $e \in \text{ supp}(\mu)$, and show that the only possible peripheral eigenvalue is $1$ in this case. Our initial proof of this Theorem was quite involved. The author would like to thank Prof.\ Jesse Peterson for the following elegant argument.
\begin{thm} \label{thm: eigenvalue for identity in support case}
	Let $\mu$ be a probability measure on $\G$, with $e \in \text{ supp}(\mu)$. Let $T \in \BB(L^2(M))$ be a nonzero operator, with $\mcp_{\mu}(T)= \lambda T$ for some $\lambda \in \mathbb T$. Then, $\lambda =1$.
\end{thm}
\begin{proof}
	As $e \in \text{ supp}(\mu)$, $\mu$ and $\mu^{\ast 2}$ are not mutually singular. Hence, by \cite[Corollary 2]{foguel} we have that $\| \mu^{ n}- \mu^{ n+1}\|_{TV} \rightarrow 0$. Hence, for any $f \in \ell^{\infty}(\G)$, we have that 
	\begin{equation} \label{eq: foguel equation convolution going to zero }
	\| f \ast \mu^n-f\ast \mu^{n+1}\|_{\infty} \rightarrow 0.
	\end{equation}
	 Suppose $0 \neq f \in \ell^{\infty}(\G)$ with $f \ast \mu =\lambda f$ for some $\lambda \in \mathbb T$. Then by equation~\ref{eq: foguel equation convolution going to zero } we get that 
	 \begin{equation} \label{eq: eigenvalue equal to 1 }
	 \|\lambda^nf-\lambda^{n+1}f\|_{\infty} \rightarrow 0 \implies |\lambda^n| \cdot |\lambda-1|\cdot \|f\|_{\infty} \rightarrow 0 \implies |\lambda-1|\cdot \|f\|_{\infty} \rightarrow 0, \text{ as }|\lambda^n|=1.
	 \end{equation}
	 As $f\neq 0$, from equation~\ref{eq: eigenvalue equal to 1 } we get that $\lambda=1$.
	 \vskip 0.01 in
	 Now let $T \in \BB(L^2(M))$ be a nonzero operator, with $\mcp_{\mu}(T)= \lambda T$ for some $\lambda \in \mathbb T$. Let $\mathcal E: \BB(L^2(M)) \rightarrow \ell^{\infty}(\G)$ denote the normal conditional expectation. Since $T \neq 0$, by Proposition~\ref{prop: eigenoperator to eigenfunction} we can find $g \in \G$ such that $f:= \mathcal E(T u_g^*) \neq 0$ and $\mcp_{\mu}(f)=\lambda f$. Thus, by the argument in the previous paragraph, we get that $\lambda =1$.
\end{proof}
We now state our main result of this section, showing that the possible peripheral eigenvalues for symmetric, generating measure can only be $\pm 1$. Note that $1$ is always a peripheral eigenvalue, as $\mcp_{\mu}$ is unital. In later parts of the paper we will show that $-1$ can also appear as an eigenvalue of $\mcp_{\mu}$, and explore the consequences.
\begin{thm} \label{thm: peripheral eigenvalue symmetric case}
	Let $\mu$ be a symmetric probability measure on $\G$. Let $T \in \BB(L^2(M))$ be a nonzero operator, with $\mcp_{\mu}(T)= \lambda T$ for some $\lambda \in \mathbb T$. Then, $\lambda= \pm1$.
\end{thm}
\begin{proof}
	Let $\nu= \mu \ast \mu$. As $\mu$ is symmetric, $e \in \text{ supp}(\nu)$. Note that $\mcp_{\nu}(T)= \mcp_{\mu}^2(T)= \lambda^2 T$. By Theorem~\ref{thm: eigenvalue for identity in support case} we get that $\lambda^2=1$. Hence, $\lambda=\pm 1$.
\end{proof}
Recently, in \cite[Example 3.4]{BTK22} Bhat, Talwar and Kar raised the problem of finding possible peripheral eigenvalues for Markov operators arising from noncommutative extension of random walks on groups as studied by Izumi in \cite{izumi4}. Recall that Izumi considered a probability measure $\mu$ on a countable, discrete group $\G$, such that the support of $\mu$ generates $\G$ as a semigroup. We now show that the only possible peripheral eigenvalues in this case can be roots of unity.
\begin{cor}
	Let $\mu$ be a probability measure  on a countable, discrete group $\G$, such that the support of $\mu$ generates $\G$ as a semigroup. Consider the noncommutative Markov operator $\mcp_{\mu}: \BB(\ell^2(\G)) \rightarrow \BB(\ell^2(\G)) $ given by $\mcp_{\mu}(T)= \sum_g \mu(g)\rho_g T \rho_g^*$. Let $\lambda \in \mathbb T$ be a peripheral eigenvalue for $\mcp_{\mu}$. Then, $\lambda^k =1$ for some positive integer $k$.
\end{cor}
\begin{proof}
	As support of $\mu$ generates $\G$ as a semigroup, there exists a smallest positive integer $k$ such that $e \in \text{ supp}(\mu^k)$. By Theorem~\ref{thm: eigenvalue for identity in support case} we get that $\lambda^k=1$.
\end{proof}
\section{Bi-harmonic operators}

In this section we provide a complete description of jointly bi-harmonic functions. This will follow from the characterization of jointly bi-harmonic operators. We first need a lemma from Revuz's book \cite[Chapter 5, Lemma 1.1]{revuz}. We reproduce the short proof for reader's convenience.

\begin{lem} \label{lem:commuting contaction}
	Let $X$ be a Banach space and let $T_1, \text{ } T_2 \in(\BB(X))_1$ be commuting contractions. Suppose there exists $0<a<1$ and $x \in X$ such that $(aT_1+ (1-a)T_2)(x)=x$. Then $T_1(x)=T_2(x)=x$.
\end{lem}

\begin{proof}
	Let $S=aT_1+ (1-a)T_2$. Using the fact that $T_1$ commutes with $T_2$, a direct calculation yields
	\begin{equation} \label{eq:revuztrick}
	\exp(-a^{-1})\exp(a^{-1}S)=\exp(-(I-T_1))\exp(1-a^{-1})\exp((a^{-1}-1)T_2).
	\end{equation}
	Let $U= \exp(1-a^{-1})\exp((a^{-1}-1)T_2)$, and note that $\|U\| \leq 1$. Taking $n^{th}$ powers on both sides of equation~\ref{eq:revuztrick} and using $S(x)=x$ we get
	\begin{align*}
	& x= \exp(-na^{-1})\exp(na^{-1}S)(x)= \exp(-n(I-T_1))U^n(x) \text{ for all }n\\
&	\implies (I-T_1)(x)= [(I-T_1)\exp(-n(I-T_1))]U^n(x) \text{ for all }n\\
&	\implies \| (I-T_1)(x) \|= \lim\limits_{n \rightarrow \infty} \|[(I-T_1)\exp(-n(I-T_1))]U^n(x)\| \leq \lim\limits_{n \rightarrow \infty} \|[(I-T_1)\exp(-n(I-T_1))]\| \cdot \|x\| = 0,
	\end{align*}
	where we used the fact that $\lim\limits_{n \rightarrow \infty} \|[(I-T_1)\exp(-n(I-T_1))]\|=0$. \vskip 0.01 in 
	To see this, first note that $\exp(-n(I-T_1))=\exp(-n)\exp(nT)=\sum\limits_{j=0}^{\infty} c_{j,n}T^j$, where $c_{j,n}=\frac{n^j}{e^nj!}$. For fixed $n$, the sequence $c_{j,n}$ increases till $j=n$ and then decreases. So we have
	$$\|[(I-T_1)\exp(-n(I-T_1))]\| \leq \sum_j |c_{j,n}-c_{j-1,n}| =\sum\limits_{j=0}^{n}(c_{j,n}-c_{j-1,n})+\sum\limits_{j=n+1}^{\infty}(c_{j-1,n}-c_{j,n})=2c_{n,n}-c_{0,n}\leq 2c_{n,n}. $$
	But $\lim\limits_{n \rightarrow \infty}c_{n,n}= \lim\limits_{n \rightarrow \infty}\frac{n^n}{e^n n!}=\lim\limits_{n \rightarrow \infty}\frac{1}{\sqrt{ 2 \pi n}}=0$, by Stirling's approximation formula. Thus we get $(I-T_1)(x)=x$, which yields $T_1(x)=x=T_2(x)$.
\end{proof}

\begin{thm} \label{thm:joint bi-harmonic1}
	Let $\nu$ be a probability measure on $\G$ such that $e \in \text{Support}(\nu)$. Let $T \in \BB(\ell^2(\G))$ such that $\mathcal P_{\nu}\circ \mcp_{\nu}^{o}(T)=T$. Then $T \in \Har(\mcp_{\nu})\cap \Har(\mcp_{\nu}^o)$. In particular, if $\nu$ is symmetric, and generating, then $T \in \mathcal Z(L(\G))$.
\end{thm}
\begin{proof} Write $\mcp_{\nu}(X)= \sum_{g \in \G} \nu(g)\rho_gX\rho_g^*= \nu(e)T+ (1-\nu(e)) \sum_{g \neq e} \frac{\nu(g)}{1-\nu(e)}\rho_gX\rho_g^*$. Let $S(X)= \sum_{g \neq e} \frac{\nu(g)}{1-\nu(e)}\rho_gX\rho_g^*$, and note that $S$ is ucp, and hence a contraction on $\BB(\ell^2(\G))$.\vskip 0.01 in
	Now let $T \in \BB(\ell^2(\G))$ such that $\mathcal P_{\nu}\circ \mcp_{\nu}^{o}(T)=T$.
	Then we have $T= \mcp_{\nu}(\mcp_{\nu}^o(T))= \nu(e)\mcp_{\nu}^o(T) +(1-\nu(e))S\circ \mcp_{\nu}^o(T)$. Note that $S$ and $\mcp_{\nu}^o$ commute, and hence $\mcp_{\nu}^o$ commutes with $S\circ \mcp_{\nu}^o$. Hence by Lemma~\ref{lem:commuting contaction} we get that $\mcp_{\nu}^o(T)=T= S \circ \mcp_{\nu}^o(T)$. \vskip 0.01 in
	
	Then from $\mathcal P_{\nu}\circ \mcp_{\nu}^{o}(T)=T$ and $\mcp_{\nu}^{o}(T)=T$ we get that $\mcp_{\nu}(T)=T$, and hence $T \in \Har(\mcp_{\nu})\cap \Har(\mcp_{\nu}^o)$. The last statement now follows from the double ergodicity theorem \cite[Theorem 3.1]{DP22}.
	\end{proof}
We now provide the characterization of bi-harmonic operators for a symmetric, generating measure (where $e$ need not lie in the support).
\begin{thm} \label{thm:joint bi-harmonic2}
	Let $\mu$ be a symmetric probability measure on $\G$. Let $T \in \BB(\ell^2(\G))$ such that $\mathcal P_{\mu}\circ \mcp_{\mu}^{o}(T)=T$. Then $T$ can be uniquely written as a sum of a (separately) bi-harmonic function, and a (separately) anti-bi-harmonic function for $\mu$. 
\end{thm}
\begin{proof}
	Let $\nu= \mu\ast \mu$. Then $e \in \text{Support}(\nu)$. We have $\mcp_{\nu}\circ \mcp_{\nu}^o(T)=(\mcp_{\mu}\circ \mcp_{\mu}^o)^2(T)=T$. By Theorem~\ref{thm:joint bi-harmonic1} we get that $\mcp_{\nu}(T)=\mcp_{\nu}^o(T)=T$, which yields $T \in \Har(\mcp_{\mu}^2) \cap (\Har(\mcp_{\mu}^o)^2)$. \vskip 0.01 in
	Note that we have the following vector space direct sum decomposition $\Har(\mcp_{\mu}^2)= \Har(\mcp_{\mu}) \oplus E_{-1}(\mcp_{\mu})$, where $E_{-1}(\mcp_{\mu})$ is the vector space of all anti-harmonic operators, i.e. $E_{-1}(\mcp_{\mu})=\{T \in \BB(\ell^2(\G)): \mcp_{\mu}(T)=-T  \}$. Indeed, if $\mcp_{\mu}^2(x)=x$, then $x= \frac{1}{2}(x+\mcp_{\mu}(x))+\frac{1}{2}(x-\mcp_{\mu}(x))$ gives a decomposition of $x$ as a sum of a harmonic and an anti-harmonic operator. As $\Har(\mcp_{\mu}) \cap E_{-1}(\mcp_{\mu})=\{0\}$, we get the direct sum decomposition as claimed.
	\vskip 0.01 in
	Hence $T \in \Har(\mcp_{\mu}^2)$ implies $T=T_0+T_1$, with $T_0 \in \Har(\mcp_{\mu})$ and $T_1 \in E_{-1}(\mcp_{\mu})$. Note that $\mcp_{\mu}^o$ preserves both the spaces $\Har(\mcp_{\mu})$ and $E_{-1}(\mcp_{\mu})$. So we get:
	$$ T= \mcp_{\mu}^o(\mcp_{\mu}(T_0 +T_1))=\mcp_{\mu}^o(T_0-T_1)= \mcp_{\mu}^o(T_0)-\mcp^o(T_1). $$
	As $\mcp_{\mu}^o(T_0) \in \Har(\mcp_{\mu})$ and $\mcp^o(T_1) \in E_{-1}(\mcp_{\mu})$, from uniqueness of decomposition we get that $\mcp_{\mu}^o(T_0)=T_0$ and $\mcp_{\mu}^o(T_1)=-T_1$. Thus, $T_0$ is separately bi-harmonic for $\mu$ and $T_1$ is separately anti-bi-harmonic for $\mu$.
\end{proof}
We can now state the answer to Kaimanovich's question from \cite{kaimanovich1}.
\begin{cor} \label{cor:joint bi-harmonic}
Let $\mu$ be a symmetric, generating probability measure on $\G$. Let $f \in \ell^{\infty}(\G)$ be a jointly $\mu$-bi-harmonic function, i.e. $\mu \ast f \ast \mu=f$. Then there exists $c \in \mathbb C$ such that $f-c$ is anti-harmonic under both left and right convolution by $\mu$.	
\end{cor}
\begin{proof}
	Let $M_{f}$ denote the multiplication operator on $\BB(\ell^2(\G))$. From $\mu \ast f \ast \mu=f$ we get $\mcp_{\mu}\circ \mcp_{\mu}^o(M_f)=M_f$. By the proof of Theorem~\ref{thm:joint bi-harmonic2} we get $T_0= \frac{1}{2}(M_{f+f \ast \mu})\in \Har(\mcp_{\mu}) \cap \Har(\mcp_{\mu}^o)$ and $T_1= \frac{1}{2}(M_{f-f \ast \mu}) \in E_{-1}(\mcp_{\mu})\cap E_{-1}(\mcp_{\mu}^o) $ with $M_f=T_0+T_1$. Since $\mu$ is generating, applying the double ergodicity theorem \cite[Theorem 3.1]{DP22}, we get that $T_0=\frac{1}{2}(M_{f+f \ast \mu})\in  \mathcal Z(L(\G))\cap \ell^{\infty}(\G)=\mathbb C$. This proves the result.  
\end{proof}
We now provide two concrete examples where jointly $\mu$-harmonic functions are not constant.  
\begin{examp} \label{eg: anti-harmonic for Z}
	Let $\G= \mathbb Z$. Consider the symmetric, generating measure $\mu$ on $\mathbb Z$ given by $\mu(1)=\mu(-1)= \frac{1}{2}$. Let $f \in \ell^{\infty}(\mathbb Z)$ be given by $f(n)=1$, if $n$ is even, and $f(n)=-1$ if $n$ is odd. Then $f \ast \mu =-f$. Since $\G$ is abelian, $\mu \ast f= f \ast \mu =-f$, and hence $\mu \ast f \ast \mu=f$. 
	 Note that the space of harmonic functions for $\mu$ just consists of constant functions by the classical Choquet-Deny Theorem.  Also, it's easy to check that all anti-harmonic functions are just constant multiples of $f$. Hence the peripheral Poisson boundary is $2$-dimensional. 
\end{examp}
\begin{examp} \label{eg: anti-harmonic for F_2}
	Let $\G= \mathbb F_2= \langle a, b \rangle$. Consider the measure $\mu$ on $\mathbb F_2$ given by $\mu(a)=\mu(b)=\mu(a^{-1})=\mu(b^{-1})$. Let $f \in \ell^{\infty}(\mathbb F_2)$ be given by $f(g)=1$, if $|g|_L$ is even, and $f(g)=-1$, if $|g|_L$ is odd. Then $\mu \ast f=-f$ and $f \ast \mu =-f$. Hence, $\mu \ast f \ast \mu =f$. So, $f$ is a (jointly) $\mu$-harmonic function, which is not constant. Note that $M_f$ is a unitary. 
\end{examp}
Given a ucp map $\Phi: M \rightarrow M$, recall that the multiplicative domain of $\Phi$ consists of all operators $m \in M$ such that $\Phi(m^*m)= \Phi(m)^*\Phi(m)$ and $\Phi(mm^*)=\Phi(m)\Phi(m)^*$. If $m$ lies in the multiplicative domain of $\Phi$, then $\Phi(ma)= \Phi(m) \Phi(a)$ and $\Phi(bm)= \Phi(b) \Phi(m)$ for all $a,b \in M$.
\vskip 0.02 in
In the next Proposition, we describe the space of anti-harmonic eigenvectors, under the assumption that there exists a unitary anti-harmonic eigenvector.
\begin{prop} \label{prop: anti harm description unitary case}
	Let $\mu$ be a symmetric, generating probability measure on $\G$. Suppose there exists a unitary $V \in \BB(\ell^2(\G))$ such that $\mcp_{\mu}(V)=-V$. Then $E_{-1}(\mcp_{\mu})=\{AV: A \in \Har(\mcp_{\mu}) \}$.	
\end{prop}

\begin{proof}
We first show that $V$ belongs to the multiplicative domain of $\mcp_{\mu}$. Indeed,  $\mcp_{\mu}(V)^*\mcp_{\mu}(V)=(-V)^*(-V)=I=\mcp_{\mu}(V^*V)$. Let $A \in \Har(\mcp_{\mu})$. Then $\mcp_{\mu}(AV)=\mcp_{\mu}(A)\mcp_{\mu}(V)-\mcp_{\mu}(A)=-A$. Thus, $AV \in E_{-1}(\mcp_{\mu})$. \vskip 0.01 in
Conversely, let $T \in E_{-1}(\mcp_{\mu})$. Let $A= TV^*$. We claim that $A \in \Har(\mcp_{\mu})$. To see this, we calculate $\mcp_{\mu}(TV^*)= \mcp_{\mu}(T)\mcp_{\mu}(V)^*=(-T)(-V)^*= TV^*=A$. So, $T=AV$, with $A \in\Har(\mcp_{\mu})$. 
\end{proof}

In the next proposition we concretely describe all separately anti-bi-harmonic operators in the scenario that we have a unitary anti-bi-harmonic operator. 

\begin{cor}\label{cor:anti biharmonic operator description unitary case}
	Let $\mu$ be a symmetric, generating probability measure on an icc group $\G$. Suppose there exists a unitary $V \in \BB(\ell^2(\G))$ such that $\mcp_{\mu}(V)=-V$, and $\mcp_{\mu}^o(V)=-V$. Let $T \in E_{-1}(\mcp_{\mu})\cap E_{-1}(\mcp_{\mu}^o)$. Then $T=cV$ for some $c \in \mathbb C$.
\end{cor}

\begin{proof}
As $T \in E_{-1}(\mcp_{\mu})$, by Proposition~\ref{prop: anti harm description unitary case} $T=AV$ for some $A\in \Har(\mcp_{\mu})$. Note that $V$ lies in the multiplicative domain of both $\mcp_{\mu}$ and $\mcp_{\mu}^o$. Since $T \in E_{-1}(\mcp_{\mu}^o)$, we have
$$-T=-AV=\mcp_{\mu}^o(T)= \mcp_{\mu}^o(AV)= \mcp_{\mu}^o(A)\mcp_{\mu}^o(V)=-\mcp_{\mu}^o(A)V.$$
So, $\mcp_{\mu}^o(A)V=AV$, which implies $\mcp_{\mu}^o(A)=A$. Hence, $A \in \Har(\mcp_{\mu})\cap \Har(\mcp_{\mu}^o)= \mathcal Z(L(\G))=\mathbb C$, by the double ergodicity theorem \cite[Theorem 3.1]{DP22}. 
\end{proof}

\section{Peripheral Eigenvectors}
Throughout this section, we will assume that $\mu$ is a symmetric, generating probability measure on $\G$. By Theorem~\ref{thm: peripheral eigenvalue symmetric case} we know that the only possible peripheral eigenvalues are $\pm 1$. We will assume that there exists a nonzero anti-$\mu$-harmonic function.
 Our goal is to prove that there exists a multiplicative character $\chi$ on $\G$ such that $\chi|_{\text{supp} (\mu)} \equiv -1$. 
 \vskip 0.01 in
 Our first result shows that we can find such an anti-harmonic character, if there exists a nonzero anti-harmonic function, which attains its maxima.
\begin{lem} \label{lem:character from anti harmonic function}
Let $f \in (\ell^{\infty}(\G))_1$ be a real valued anti-harmonic function such that $f(e)=1$, then $f$ is a character on $\G$ with $f|_{\text{supp} (\mu)} \equiv -1.$
\end{lem}
\begin{proof}
	Let $S$ denote the support of $\mu$. As $f$ is anti-harmonic, we get \begin{equation} \label{eq:antiharmonic at identity}
	-1=-f(e)= (f \ast \mu)(e)= \sum_{h \in S} \mu(h)f(h).
	\end{equation}
	As $-1 \leq f(\g) \leq 1$ for all $\g \in \G$, equation~\ref{eq:antiharmonic at identity} implies that $f(h)=-1$ for all $h \in \text{supp}(\mu)$. Now let $g \in \text{supp}(\mu)$. Using $f(g)=-1$, and anti-harmonicity of $f$ we get
	\begin{equation} \label{eq:antiharmonic at group element}
	1=-f(g)=(f \ast \mu)(g)= \sum_{h \in S} \mu(h)f(gh).
	\end{equation}
	Equation~\ref{eq:antiharmonic at group element} yields $f(gh)=1=f(g)f(h)$ for all $g,h \in \text{supp}(\mu)$. Inductively, we get that $f(h_1h_2\ldots h_n)=f(h_1)f(h_2)\ldots f(h_n)$ for all $h_1,h_2,\ldots ,h_n \in \text{supp}(\mu)$. As $\text{supp}(\mu)$ generates $\G$ as a semigroup (being a symmetric set, this follows from the generating assumption), we get that $f$ is a character on $\G$.
\end{proof}
We will now deduce the existence of an anti-harmonic function satisfying the conditions of Lemma~\ref{lem:character from anti harmonic function}.
Consider the minimal dilation $\theta$ of $\mcp_{\mu}: \ell^{\infty}(\G) \rightarrow \ell^{\infty}(\G)$. The peripheral Poisson boundary can be identified with $\Har(\theta^2)$, and is an abelian von Neumann algebra. Note that the classical Poisson boundary is identified with $\Har(\theta)$, and is a von Neumann subalgebra of $\Har(\theta^2)$. We will now prove a technical result regarding existence of a self-adjoint unitary eigenvector for $\theta$ corresponding to the eigenvalue $-1$.
\begin{prop} \label{prop: self adjoint unitary anti-harmonic}
Let $\mu$ be as above. Then there exists a self adjoint unitary $u \in \Har(\theta^2)$ with $\theta(u)=-u$.
\end{prop}

\begin{proof}
	Suppose first that the Poisson boundary is trivial. Let $0 \neq x \in \Har(\theta^2)$ with $\theta(x)=-x$, and $x=x^*$. Then, $\theta(x^2)= \theta(x)^2=x^2$. So, $x^2 \in \Har(\theta)=\mathbb C$. After normalization, we may assume that $x$ is a self-adjoint unitary.  \vskip 0.01 in
	We now assume that the Poisson boundary is nontrivial. Let $p \in \text{Proj}(\Har(\theta^2))$ with $\theta(p) \neq p$. We decompose $p=x_1+x_2$ with $\theta(x_1)=x_1$ and $\theta(x_2)=-x_2$. Hence $\theta(p)=x_1-x_2$.
	Note that $x_1$ and $x_2$ are self-adjoint, and $x_1$ s positive. Also, $x_2 \neq 0$, as $\theta(p) \neq p$. \vskip 0.01 in
	 From $p=p^2$ and $\theta(p)^2=\theta(p)$ we get
	 \begin{eqnarray}
	 x_1^2+ 2x_1x_2+x_2^2=x_1+x_2.\\
	  x_1^2- 2x_1x_2+x_2^2=x_1-x_2.
	 \end{eqnarray}
	 The above two equations yield 
	 \begin{eqnarray} \label{eq: x_1 x_2 relation}
	 x_1=x_1^2+x_2^2 \text{ and }
	 x_2=2x_1x_2.
	 \end{eqnarray}
	 As $\Har(\theta^2)$ is an abelian $C^*$-algebra, by Gelfand duality, we can find an extremely disconnected space $X$ such that $C(X) \cong \Har(\theta^2)$. We will denote the image of any element $y \in \Har(\theta^2)$ as $\hat{y} \in C(X)$ under this isomorphism. We will also denote the pre-image of $f \in C(X)$ as $\check{f} \in \Har(\theta^2)$. Now let $f = \hat p$, $f_1 = \hat x_1$ and $f_2= \hat x_2$. From equation~\ref{eq: x_1 x_2 relation} we get that
	 \begin{eqnarray} 
	 f_1=f_1^2+f_2^2 \text{ and } \label{eq: f_1 f_2 relation1} \\
	 f_2=2f_1f_2. \label{eq: f_1 f_2 relation 2}
	 \end{eqnarray}
	 Let $a \in X$ with $f_2(a) \neq 0$. From equation~\ref{eq: f_1 f_2 relation 2} we get $f_1(a)=\frac{1}{2}$. Equation~\ref{eq: f_1 f_2 relation1} then gives $f_2(a)=\pm \frac{1}{2}$. \vskip 0.01 in
	 Now let $b \in X$ with $f_2(b)=0$. Then from equation~\ref{eq: f_1 f_2 relation1}, we get $f_1(b)=f_1(b)^2$, which yields $f_1(b)=0$ or $f_1(b)=1$. \vskip 0.01 in
	 Let $A=\{a \in X: f_2(a) \neq 0 \}$, $A_1=\{a \in X: f_2(a)=\frac{1}{2} \}$ and $A_2=\{a \in X: f_2(a)=-\frac{1}{2} \}= A \setminus A_1$. \vskip 0.01 in
	 Let $B=X \setminus A=\{b \in X: f_2(b)=0 \}$. Let $B_1=\{b \in B: f_1(a)=1 \}$ and $B_2=B \setminus B_1= \{b \in B: f_1(a)=0 \}.$ \vskip 0.01 in
	 Then, $f_2=\frac{1}{2}\chi_{A_1}-\frac{1}{2}\chi_{A_2}$ and $f_1=\frac{1}{2}\chi_{A}+ \chi_{B_1}= \frac{1}{2}\chi_{A_1}+ \frac{1}{2}\chi_{A_2}+ \chi_{B_1}$. Hence, $\hat p= f =f_1+f_2= \chi_{A_1}+ \chi_{B_1}$ and $\widehat{\theta(p)}= f_1-f_2= \chi_{A_2}+ \chi_{B_1}$. Let $p_1= \check{\chi_{B_1}}$,  $p_2= \check{\chi_{B_1}}$, and  $p_3= \check{\chi_{B_3}}$. Note that $p_1$, $p_2$ and $p_3$ are mutually orthogonal projections, with 
	 \begin{equation} \label{eq:p theta(p) in terms of p_1, p_2, p_3}
	 p=p_1 + p_2 \text{ and } \theta(p)=p_1 + p_3. 
	 \end{equation}	 
	 We also note
	 \begin{eqnarray} 
	 p \theta(p)=(p_1+p_2)(p_1+p_3)= p_1, \text{ which implies }
	  \theta(p_1)= \theta(p \theta(p))= \theta(p) \theta^2(p)= \theta(p)p=p_1. \label{eq:harmonicity of p_1}	  
	 \end{eqnarray}
	 Hence, $p_1$ is harmonic.  Also, note that, as $p \neq \theta(p)$ by choice, we have $p_2 \neq 0$. From equation~\ref{eq:p theta(p) in terms of p_1, p_2, p_3} we get that $\theta(p_2)=p_3$, which implies $\theta(p_2) \leq p_2^{\perp}$. \vskip 0.01 in
	 Let $q$ denote the maximal projection such that $\theta(q) \leq q^{\perp}$. By above argument, $q \neq 0$. We will now argue that $\theta(q)= q^{\perp}$. Assume for the sake of contradiction that $p_0= q^{\perp}- \theta(q) \neq 0$. Note that $\theta(p_0)= \theta(1-q)-\theta^2(q)= 1-\theta(q)-q= q^{\perp}-\theta(q)=p_0$. So, $p_0$ is harmonic. Let $q_1 \leq p_0$ be any subprojection. Arguing as above, we can find mutually orthogonal projections $q_{1,1}$, $q_{1,2}$ and $q_{1,3}$ such that
	 \begin{eqnarray}
	 q_1= q_{1,1}+ q_{1,2}\\
	 \theta(q_1)= q_{1,1}+ q_{1,3} \\
	 \theta(q_{1,1})= q_{1,1}, \text{ and } \theta(q_{1,2})= q_{1,3}. 
	 \end{eqnarray}
	 Note that $q_{1,i} \leq q_1 \leq p_0 \leq q^{\perp}$ for $i=1,2,3$. We now compute
	 \begin{align*}
	 (q+q_{1,2}) \theta( (q+q_{1,2}))&=  (q+q_{1,2})(\theta(q)+ q_{1,3})) = q\theta(q)+ q \cdot q_{1,3}+ q_{1,2} \theta(q) + q_{1,2} \cdot q_{1,3} \\&= 0+ 0+ q_{1,2} \theta(q)+0=0,
	 \end{align*}
	 where the last equality follows from $q_{1,2} \theta(q) = q_{1,2}(q^{\perp}-p_0)= q_{1,2}-q_{1,2}$, as $q_{1,2}$ is a subprojection of both $p_0$ and $q^{\perp}$. In conclusion, we have shown that $(q+q_{1,2}) \theta( (q+q_{1,2}))=0$, which implies $q+q_{1,2} \leq (q+q_{1,2})^{\perp}$. This will contradict the maximality of $q$, unless $q_{1,2}=0$. As $q_{1,3}= \theta(q_{1,2})$ we get that $q_{1,3}=0$ as well. Hence, $q_1$ is harmonic. \vskip 0.01 in
	 The above paragraph shows that every subprojection of $p_0$ is harmonic. Hence we must have $\theta(p_0 u_gqu_g^*)= p_0 u_gqu_g^*$ for all $g \in \G$. This then implies
	 \begin{equation} \label{eq: conjugation of q time p_0}
	 p_0u_g(\theta(q)-q) =0 \text{ for all } g \in \G.
	 \end{equation}
	 Multiplying both sides of equation~\ref{eq: conjugation of q time p_0} by $u_g^*$, we get \begin{equation} \label{eq: conjugate of p_0 times theta q minus q}
	  (u_g^* p_0 u_g)(\theta(q)-q) =0 \text{ for all } g \in \G.
	 \end{equation}
	 Now $\G \curvearrowright \Har(\theta)$ is SAT \cite{jaworski1}, and $0 \neq p_0 \in \Har(\theta)$ by assumption. Hence we can find a sequence $\{g_n \}$ such that $u_{g_n}^* p_0 u_{g_n} \rightarrow 1$ in SOT (see \cite{creutzpetersonchar}). This implies via equation~\ref{eq: conjugate of p_0 times theta q minus q} that $\theta(q)=q$. However, $\theta(q) \leq q^{\perp}$, and hence $\theta(q)=q$ would imply $q=0$. This is a contradiction, as we have established earlier that $q \neq 0$. The contradiction arose from the assumption that $p_0 \neq 0$. Hence $p_0=0$, which implies $\theta(q)=q^{\perp}$. \vskip 0.01 in
	 Now let $u=2q-1$. Clearly $u$ is a self-adjoint unitary in $\Har(\theta^2)$. Furthermore, $\theta(u)= 2 \theta(q)-1= 2(1-q)-1=-(2q-1)=-u$. Hence we are done. 
	 
\end{proof}

\begin{prop}\label{prop: extreme point of anti harmonic function}
	Let $\mu$ be as above. Then there exists an anti-$\mu$-harmonic function $f \in (\ell^{\infty}(\G))_1$ such that $\mcp_{\mu}^n(|f|)\nnearrow 1$.
	\end{prop}
\begin{proof}
	By Proposition~\ref{prop: self adjoint unitary anti-harmonic} we can find a self-adjoint unitary $v \in \Har(\theta^2)$, with $v^2=1$. Let $E_{-1}(\mu)$ denote the eigenspace corresponding to the eigenvalue $-1$, and $E_1$ be the eigenspace corresponding to the eigenvalue $1$ (i.e. the space of harmonic functions). Bhat, Talwar, and Kar showed that $\tilde{\BB_{\mu}}:=\overline{\text{span}}\{f \in \ell^{\infty}(\G) \text{ with } f \ast \mu =\pm f \}$ is a $C^*$-algebra under a new Choi-Effros type product defined as follows: \vskip 0.02 in	
	If $f_1, \in E_{\lambda_1}(\mu)$, and $f_2 \in E_{\lambda_2}(\mu)$ where $\lambda_1, \lambda_2 \in \{\pm1 \} $, then $f_1 \diamond f_2= s-\lim (\lambda_1 \lambda_2)^{-n} \mcp_{\mu}^n(f_1 f_2)$. Note that $\tilde{\BB_{\mu}}$ is an abelian $C^*$-algebra under this product. \vskip 0.02 in	
	 In fact, by \cite[Corollary 4.4]{BTK22} we get that $\tilde{\BB_{\mu}}$ is a von Neumann algebra under this product. Moreover, note that $\tilde{\BB_{\mu}}$ is exactly $\Har(\theta^2)$. Let $f \in \ell^{\infty}(\G)$ be a real valued anti-$\mu$-harmonic function such that $f$ corresponds to the self-adjoint unitary $v$ under the above identification.
	 \vskip 0.02 in	
	 So, we have that $s-\lim \mcp_{\mu}^n(f^2)=1$. Note that $f^2= \mcp_{\mu}(f)^*\mcp_{\mu}(f) \leq \mcp_{\mu}(f^2)$ by Kadison-Schwarz inequality. An easy induction argument now implies that the sequence $\{\mcp_{\mu}^n(f^2) \}$ is increasing. So, $\mcp_{\mu}^n(f^2) \nnearrow 1$. \vskip 0.02 in
	 Note that $\mcp_{\mu}(|f|) \geq |\mcp_{\mu}(f)|=|-f|=|f|$. By induction, we get $\{\mcp_{\mu}^n(|f|) \}$ is an increasing sequence. Also, as $\|f\|_{\infty} \leq 1$, we have $f^2 \leq |f|$, which implies that $ \mcp_{\mu}^n(f^2) \leq \mcp_{\mu}^n(|f|) \leq 1$. Hence we get $\mcp_{\mu}^n(|f|)\nnearrow 1$.
\end{proof}
\begin{thm} \label{thm: anti-harmonic character}
	Let $f \in (\ell^{\infty}(\G))_1$ be a real valued anti-$\mu$-harmonic function such that $\mcp_{\mu}^n(|f|)\nnearrow 1$. Then there exists an anti-$\mu$-harmonic character $\chi \in \ell^{\infty}(\G)$, with $\chi|_{\text{supp}(\mu)}\equiv -1$.
\end{thm}
\begin{proof}
	By assumption we have \begin{equation} \label{eq: limit of absolute value of anti harmomic extreme point}
	1= s-\lim \mcp_{\mu}^n(|f|)(e)= \sum_{g \in \G}\mu^{\ast n}(g) |f(g)|.  
	\end{equation}
Let $\varepsilon_1=\frac{1}{2}$. As $\mu^{\ast n}$ is a probability measure, from equation~\ref{eq: limit of absolute value of anti harmomic extreme point}, there exists $n_1 \in \mathbb N$ and $g_1 \in \text{supp}(\mu^{\ast n_1})$ such that $|f|(g_1) > 1 -\varepsilon_1$. Replacing $f$ by $-f$ if necessary, we may assume $f(g_1)>1 -\varepsilon_1$. Consider $f_1(x)=f(g_1x)$ for all $x \in \G$. Then $f_1(e)> 1 -\varepsilon_1$. Also, $(f_1 \ast \mu)(x)=\sum_{h \in \G} \mu(h)f_1(xh)= \sum_{h \in \G} \mu(h)f(g_1xh)=-f(g_1x)=-f_1(x)$. So, $f_1$ is anti-harmonic. Also, 
\begin{equation} \label{eq: iteration step 1}
\mcp_{\mu}^n(|f_1|)(x)= \sum_{h \in \G}\mu^{\ast n}(h) |f(g_1xh)| \nnearrow 1,
\end{equation} 
where the limit is in the strong-operator topology.  Let $\varepsilon_2=\frac{1}{3}$. Using the same argument, we deduce the existence of an anti-harmonic function $f_2$ with $f_2(e)> 1- \varepsilon_2$, and $\mcp_{\mu}^n(|f_2|)(x) \nnearrow 1$ for all $x \in \G$. Inductively, we get a sequence of real valued functions $f_k \in \ell^{\infty}(\G)$ satisfying the following properties:
\begin{enumerate}
	\item[i)] $\|f_k\|_{\infty} \leq 1$.
	\item[ii)] $f_k \ast \mu =-f_k$.
	\item[iii)] $f_k(e) > 1-\varepsilon_k$, where $\varepsilon_k= \frac{1}{k+1}$.
	\item[iv)] $\mcp_{\mu}^n(|f_k|)(x) \nnearrow 1$ for all $x \in \G$, where the limit is in the strong-operator topology.
\end{enumerate}
Let $f_{k_j}$ be a subsequence of $f_k$ that converges pointwise and let $\chi \in \ell^{\infty}(\G)$ be the pointwise limit of this subsequence. Hence, $\|\chi\|_{\infty} \leq 1$, and $\chi(e)= \lim_{j \rightarrow \infty} f_{k_j}(e) =1$. Also, we have 
\begin{equation}\label{eq: anti harmonicity of chi}
\langle \mcp_{\mu}(M_{\chi}) \delta_g, \delta_h \rangle = \lim \langle\mcp_{\mu}( M_{f_{k_j}}) \delta_g, \delta_h \rangle =-\lim \langle M_{f_{k_j}} \delta_g, \delta_h \rangle= -\langle M_{\chi} \delta_g, \delta_h \rangle.
\end{equation}
Hence $\mcp_{\mu}(M_{\chi})=-M_{\chi}$ which implies $\chi \ast \mu = -\chi$. As $\chi(e)=1$, by Lemma~\ref{lem:character from anti harmonic function}, we get that $\chi$ is a character on $\G$ with $\chi|_{\text{supp}(\mu)}\equiv -1$.
\end{proof}
We end this section with the following easy observation.
\begin{cor}
	Let $\G$ be a countable, discrete group with no subgroup of index $2$. Then there does not exist a nonzero anti-harmonic function for any symmetric, generating measure $\mu$ on $\G$. Hence, for any symmetric, generating measure $\mu$ on $\G$, every jointly bi-harmonic function is constant.
\end{cor}
\begin{proof}
If $\G$ admits a nonzero anti-harmonic function, then $\G$ must admit a surjective homomorhism $\chi$ onto $\mathbb Z_2$ by Theorem~\ref{thm: anti-harmonic character}. Then Ker$(\chi)$ is an index $2$ subgroup of $\G$, which establishes the result.
\end{proof}

\section*{Acknowledgments}  
The author is indebted to Prof.\ Jesse Peterson for many valuable comments and suggestions regarding this paper, and for his help and support. The author is very grateful to Prof.\ Ionu\c t Chifan, Prof.\ Raul Curto, Prof.\ Palle Jorgensen, and Prof.\ Paul Muhly for their support and encouragement. Part of this work was done while the author was visiting Vanderbilt University, and The University of Iowa. The author is very thankful to these institutions for their hospitality.

\def\cprime{$'$}
\providecommand{\bysame}{\leavevmode\hbox to3em{\hrulefill}\thinspace}
\providecommand{\MR}{\relax\ifhmode\unskip\space\fi MR }
\providecommand{\MRhref}[2]{%
  \href{http://www.ams.org/mathscinet-getitem?mr=#1}{#2}
}
\providecommand{\href}[2]{#2}

\vskip 0.01 in
\noindent \textsc{Department of Mathematics, Embry-Riddle Aeronautical University, 3700 Willow Creek Road, Prescott, AZ 86301, USA. }

\noindent \email {sayan.das@erau.edu}
\end{document}